\documentclass[a4paper]{article}
\usepackage{amsmath,amsthm,amssymb,enumerate,mathrsfs}
\newtheorem{dfn}{Definition}[section]
\newtheorem{thm}[dfn]{Theorem}

\newtheorem{lem}[dfn]{Lemma}

\newtheorem{rem}[dfn]{Remark}

 \makeatletter
 \@addtoreset{equation}{section} 
\makeatother
\newcommand{\N}{\mathbb{N}}

\newcommand{\R}{\mathbb{R}}

\newcommand{\A}{\mathcal{A}}

\newcommand{\Hom}{\mathop{\mathrm{Hom}}\nolimits}

\def\Xint#1{\mathchoice
 {\XXint\displaystyle\textstyle{#1}}%
 {\XXint\textstyle\scriptstyle{#1}}%
 {\XXint\scriptstyle\scriptscriptstyle{#1}}%
 {\XXint\scriptscriptstyle\scriptscriptstyle{#1}}%
 \!\int} \def\XXint#1#2#3{{\setbox0=\hbox{$#1{#2#3}{\int}$}
 \vcenter{\hbox{$#2#3$}}\kern-.5\wd0}} 
\def\dashint{\Xint{-}}
\makeatletter
 \def\@maketitle{% 
\begin{flushright}% 
{\large \@date}%  
\end{flushright}% 
\par\vskip 1.5em
\begin{center}% 
{\LARGE \@title \par}%  
\end{center}% 
\begin{center}% 
{\large \@author}% 
\end{center}% 
\par\vskip 1.5em
 }
\title{Partial regularity for elliptic systems with VMO-coefficients}
\author{Taku Kanazawa \\
Graduate School of Mathematics \\
Nagoya University, JAPAN
}
\date{Ver. 5 July 2013
}
\begin{document}
\maketitle

\begin{center}
\begin{minipage}[t]{10cm}
%-- first page ------------------------- Abstract
\small{
\noindent \textbf{Abstract.}
We establish partial H\"older continuity for vector-valued solutions $u:\Omega\to\R^N$ to inhomogeneous elliptic systems
of the type:
\[
 -\mathrm{div}(A(x,u,Du))=f(x,u,Du) \qquad \mathrm{in}\>\Omega ,
\]
where the coefficients $A:\Omega\times\R^N\times\Hom(\R^n,\R^N)\to\Hom(\R^n,\R^N)$ are possibly discontinuous with respect to 
$x$. More precisely, we assume a VMO-condition with respect to the $x$ and continuity with respect to $u$ and prove 
H\"older continuity of the solutions outside of singular sets.
\medskip

%-- first page ------------------------- Keywords
\medskip

%-- first page ------------------------- AMS classification codes
\noindent \textbf{Mathematics~Subject~Classification~(2010):}
35J60, 35B65.

}
\end{minipage}
\end{center}
\section{Introduction}
In this paper, we consider the second order nonlinear elliptic systems in divergence form of the following type:
\begin{equation}
 -\mathrm{div}(A(x,u,Du))=f(x,u,Du) \qquad \mathrm{in}\>\Omega .
 \label{system}
\end{equation}
Here $\Omega$ is bounded domain in $\R^n$, $u$ takes values in $\R^N$ with coefficients 
$A:\Omega\times\R^N\times\Hom(\R^n,\R^N)\to\Hom(\R^n,\R^N)$. 

The aim of this paper is to obtain a partial regularity result of weak solutions to \eqref{system}
with discontinuous coefficients. More precisely, we assume that the partial mapping
$x\mapsto A(x,u,\xi)/(1+|\xi|)^{p-1}$ has vanishing mean oscillation (VMO), uniformly in $(u,\xi)$. This means that
$A$ satisfies an estimate
\[
 |A(x,u,\xi)-(A(\cdot,u,\xi))_{x_0,\rho}|\leq V_{x_0}(x,\rho)(1+|\xi|)^{p-1},
\]
where $V_{x_0}:\R^n\times[0,\rho_0]\to[0,2L]$ are bounded functions with 
\[
 \lim_{\rho\searrow 0}V(\rho)=0, \quad 
 V(\rho):=\sup_{x_0\in\Omega}\sup_{0<r\leq\rho}\dashint_{B_r(x_0)\cap\Omega}V_{x_0}(x,r)dx.
\]
We also assume that $u\mapsto A(x,u,\xi)/(1+|\xi|)^{p-1}$ is continuous, that is, there exists a modulus of 
continuity $\omega:[0,\infty)\to[0,\infty)$ such that an estimate
\[
 |A(x,u,\xi)-A(x,u_0,\xi)|\leq L\omega(|u-u_0|^2)(1+|\xi|)^{p-1}
\]
holds. 

Regularity results under a VMO-condition have been established by Zheng \cite{Zheng} for 
quasi-linear elliptic systems or integral functionals. General functionals with VMO-coefficients were 
consider by Ragusa and Tachikawa \cite{RT}, who generalized the low-dimensional results from problems with
continuous coefficients to the case of VMO-coefficients. In particular, these results require that  
the dimension of domain is small, for example, $n\leq p+2$ is required to obtain the H\"{o}lder continuity of 
the minimizers in \cite{RT}. In contrast, B\"{o}gelein, Duzaar, Habermann and Scheven
\cite{BDHS} give the regularity result for homogeneous nonlinear elliptic system without dimension conditions. 

Stronger assumptions such as the H\"{o}lder continuity with respect to $(x,u)$ or a Dini-type condition
lead to partial $C^1$-regularity with a quantitative modulus of continuity for $Du$; the modulus of continuity can 
be determined in dependence on the modulus of continuity of the coefficients (cf.  
Giaquinta and Modica \cite{GM2}, Duzaar and Grotowski \cite{DG}, Duzaar and Gastel \cite{DGa}, Chen and Tan \cite{CT},
Qiu \cite{Qiu} and the references therein). 

As we knew, we could not expect continuity (and not even boundedness) of the gradient $Du$ under continuous coefficients (or 
even more relaxed condition). The regularity result with continuous coefficients was already proved in eighties by Campanato 
\cite{Cam2,Cam3}. He proves that we could still expect local H\"older continuity of the solution $u$ in special cases, 
for instance, in lower dimension $n\leq p+2$. The result with arbitrary dimension was given by 
Foss and Mingione \cite{FM}. 

Our aim is to extend the homogeneous system result in \cite{BDHS} to inhomogeneous system. Therefore 
we assume the same structure conditions to coefficients $A$ as in \cite{BDHS}. Under a suitable assumption to 
inhomogeneous term, we obtain H\"{o}lder continuity of weak solution (see Theorem \ref{pr}).

Our proof is based on so-called $\A$-harmonic approximation (cf. \cite[Lemma 2.1]{DG}; see also Lemma \ref{A-harm}), 
introduced by Duzaar and Grotowski. They give a simplified (direct) proof of regularity results to the systems 
with H\"{o}lder continuous coefficients and a natural growth condition, without $L^p$-$L^2$-estimates for $Du$. 

We close this section by briefly summarizing the notation used in this paper.
As mentioned above, we consider a bounded domain $\Omega\subset\R^n$, and maps from $\Omega$ to $\R^N$,
where we take $n\geq 2$, $N\geq 1$. For a given set $X$ we denote by $\mathscr{L}^n(X)$ the $n$-dimensional 
Lebesgue measure. We write $B_\rho(x_0):=\{ x\in\R^n\> :\> |x-x_0|<\rho\}$. For bounded set 
$X\subset\R^n$ with $\mathscr{L}^n(X)>0$, we denote the average of a given function $g\in L^1(X,\R^N)$ by 
$\dashint_Xgdx$, that is, $\dashint_Xgdx =\frac{1}{\mathscr{L}^n(X)}\int_Xgdx$. In particular, we write
$g_{x_0,\rho}=\dashint_{B_\rho(x_0)\cap\Omega}gdx$. We write $\mathrm{Bil}(\mathrm{Hom}(\R^n,\R^N))$
for the space of bilinear forms on the space $\mathrm{Hom}(\R^n,\R^N)$ of linear maps from $\R^n$ to
$\R^N$. We denote $c$ a positive constant, possibly varying from line by line. Special occurrences will be denoted by 
capital letters $K$, $C_1$, $C_2$ or the like.  

\section{Statement of the results}
\begin{dfn}\label{wsol}
We say $u\in W^{1,p}(\Omega,\R^N)$, $p\geq 2$ is a weak solution of \eqref{system} if $u$ satisfies
\begin{equation}
 \int_\Omega \langle A(x,u,Du),D\varphi\rangle dx
 =\int_\Omega \langle f,\varphi\rangle dx \label{ws}
\end{equation}
for all $\varphi\in C^{\infty}_0(\Omega,\R^N)$, where $\langle\cdot,\cdot\rangle$ is the standard Euclidean 
inner product on $\R^N$ or $\R^{nN}$.
\end{dfn}
We assume following structure conditions.
\begin{enumerate}
\item[({\bf H1})]
 $A(x,u,\xi)$ is differentiable in $\xi$ with continuous derivatives, that is, there exists
 $L\geq 1$ such that
\begin{equation}
 \left| A(x,u,\xi)\right|+(1+|\xi|)
 \left| D_\xi A(x,u,\xi) \right|\leq L(1+|\xi|)^{p-1}
\end{equation}
for all $x\in\Omega$, $u\in\R^N$ and $\xi\in\mathrm{Hom}(\R^n,\R^N)$. Moreover, from this we deduce the modulus of continuity 
function $\mu:[0,\infty)\to [0,\infty)$ such that $\mu$ is bounded, concave, non-decreasing and we have
\begin{equation}
 \left| D_\xi A(x,u,\xi)-D_\xi A(x,u,\xi_0)\right|
 \leq L\mu\left( \frac{|\xi-\xi_0|}{1+|\xi|+|\xi_0|}\right)
 (1+|\xi|+|\xi_0|)^{p-2}
\end{equation}
for all $x\in\Omega$, $u\in\R^N$, $\xi,\xi_0\in\mathrm{Hom}(\R^n,\R^N)$. Without loss of generality, we may assume $\mu\leq 1$.
\item[({\bf H2})]
 $A(x,u,\xi)$ is uniformly strongly elliptic, that is, for some $\lambda>0$ we have
\begin{equation}
 \biggl\langle D_\xi A(x,u,\xi)\nu,\nu \biggr\rangle
 :=\sum_{\substack{1\leq i,\beta\leq N\\ 1\leq j,\alpha\leq n}}
 D_{\xi_\beta^j}A_\alpha^i(x,u,\xi)\nu_i^\alpha\nu_j^\beta
 \geq \lambda |\nu|^2(1+|\xi|)^{p-2}
\end{equation}
for all $x\in\Omega$, $u\in\R^N$, $\xi,\nu\in\mathrm{Hom}(\R^n,\R^N)$. 
\item[({\bf H3})]
 $A(x,u,\xi)$ is continuous with respect to $u$. There exists a bounded, concave and non-decreasing function 
$\omega:[0,\infty)\to [0,\infty)$ satisfying  
\begin{equation}
 |A(x,u,\xi)-A(x,u_0,\xi)|\leq L\omega\left( |u-u_0|^2\right)
 (1+|\xi|)^{p-1}
\end{equation}
for all $x\in\Omega$, $u,u_0\in\R^N$, $\xi\in\mathrm{Hom}(\R^n,\R^N)$. Without loss of generality, we may assume $\omega\leq 1$.
\item[({\bf H4})]
 $x\mapsto A(x,u,\xi)/(1+|\xi|)^{p-1}$ fulfils the following VMO-conditions uniformly in $u$ and $\xi$:
\[
 | A(x,u,\xi)- \left( A(\cdot ,u,\xi)\right)_{x_0 ,\rho} | \leq
 V_{x_0}(x,\rho)(1+|\xi|)^{p-1}, \qquad \text{for all }x\in B_{\rho}(x_0)
\]
whenever $x_0\in\Omega$, $0<\rho<\rho_0$, $u\in\R^N$ and $\xi\in\Hom(\R^n,\R^N)$, where $\rho_0>0$ and 
$V_{x_0}:\R^n\>\times\> [0,\rho_0]\to[0,2L]$ are bounded functions satisfying 
\begin{equation}
 \lim_{\rho\searrow 0}V(\rho)=0, \qquad V(\rho):=\sup_{x_0\in
 \Omega}\sup_{0<r\leq\rho}\dashint_{B_r(x_0)\cap\Omega}V_{x_0}(x,r)dx.
\end{equation}
\item[({\bf H5})]
 $f(x,u,\xi)$ has $p$-growth, that is, there exist constants $a,b\geq 0$, with $a$ possibly depending on $M>0$,
such that 
\begin{equation}
 |f(x,u,\xi)|\leq a(M)|\xi|^p +b
\end{equation}
for all $x\in\Omega$, $u\in\R^N$ with $|u|\leq M$ and $\xi\in\mathrm{Hom}(\R^n,\R^N)$.
\end{enumerate}

Now, we are ready to state our main theorem.

\begin{thm}\label{pr}
Let $u\in W^{1,p}(\Omega,\R^N)\cap L^{\infty}(\Omega,\R^N)$ be a bounded weak solution of \eqref{system} under the 
structure conditions {\rm ({\bf H1}), ({\bf H2}), ({\bf H3}), ({\bf H4})} and {\rm ({\bf H5})} satisfying 
$\|u\|_{\infty}\leq M$ and $2^{(10-9p)/2}\lambda>a(M)M$.
Then there exists an open set $\Omega_u\subseteq\Omega$ with $\mathscr{L}^n(\Omega\setminus\Omega_u)=0$ such that
$u\in C^{0,\alpha}_{\mathrm{loc}}(\Omega_u,\R^N)$ for every $\alpha\in(0,1)$. Moreover, we have 
$\Omega\setminus\Omega_u\subseteq\Sigma_1\cup\Sigma_2$, where
\begin{align*}
 \Sigma_1:&=\left\{ x_0\in\Omega \>\> :\>\>
 \liminf_{\rho\searrow 0}\dashint_{B_\rho(x_0)}|Du-(Du)_{x_0,\rho}|^pdx
 >0\right\}, \\
 \Sigma_2:&=\left\{ x_0\in\Omega \>\> :\>\>
 \limsup_{\rho\searrow 0}|(Du)_{x_0,\rho}|=\infty\right\}.
\end{align*}
\end{thm}
\section{Preliminaries}
In this section we present $\A$-harmonic approximation lemma and some standard estimates 
for the proof of the regularity theorem. 

First we state the definition of $\A$-harmonic function and recall $\A$-harmonic approximation lemma as below.

\begin{dfn}[{\cite[Section 1]{DG}}]
For a given $\A\in\mathrm{Bil}(\mathrm{Hom}(\R^n,\R^N))$, we say that $h\in W^{1,p}(\Omega,\R^N)$ is 
an $\A$-harmonic function, if $h$ satisfies
\[
 \int_{\Omega}\A(Dh,D\varphi)dx=0
\]
for all $\varphi\in C^\infty_0(\Omega,\R^N)$.
\end{dfn}

\begin{lem}[{\cite[Lemma 2.3]{BDHS}}]\label{A-harm}
Let $\lambda>0$, $L>0$, $p\geq 2$ and $n,N\in\N$ with $n\geq 2$ given. For every $\varepsilon>0$,
there exists a constant $\delta=\delta(n,N,L,\lambda,\varepsilon)\in(0,1]$ such that the following holds:
assume that $\gamma\in [0,1]$ and $\A\in\mathrm{Bil}(\mathrm{Hom}(\R^n,\R^N))$ with the property 
\begin{align}
 \mathcal{A}(\nu,\nu)&\geq\lambda|\nu|^2,\qquad
 \text{for all }\nu\in\mathrm{Hom}(\R^n,\R^N), \\
 \mathcal{A}(\nu,\tilde{\nu})&\leq L|\nu||\tilde{\nu}|,\qquad
 \text{for all }\nu,\tilde{\nu}\in\mathrm{Hom}(\R^n,\R^N).
\end{align}
Furthermore, let $g\in W^{1,2}(B_\rho(x_0),\R^N)$ be an approximately $\A$-harmonic map in sense that there holds
\begin{align}
 \dashint_{B_\rho(x_0)}&\left\{ |Dg|^2+\gamma^{p-2}|Dg|^p\right\} dx\leq 1, \\
 \left|\dashint_{B_\rho(x_0)}\A(Dg,D\varphi)dx\right|
 &\leq\delta\sup_{B_\rho(x_0)}|D\varphi|,\quad \text{for all }\varphi
 \in C^1_c(B_\rho(x_0),\R^N).
\end{align}
Then there exists an $\A$-harmonic function $h$ that satisfies
\begin{align}
 \dashint_{B_\rho(x_0)}&\left\{\left|\frac{h-g}{\rho}\right|^2
 +\gamma^{p-2}\left|\frac{h-g}{\rho}\right|^p\right\} dx\leq\varepsilon , \\
 \dashint_{B_\rho(x_0)}&\left\{|Dh|^2
 +\gamma^{p-2}|Dh|^p\right\} dx\leq c(n,p).
\end{align}
\end{lem}
Next is a standard estimates for the solutions to homogeneous second order elliptic systems with
constant coefficients, due originally to Campanato \cite[Teorema 9.2]{Cam}. For convenience, we state the estimate in 
a slightly general form than the original one.
\begin{thm}[{\cite[Theorem 2.3]{DG}}]\label{Campanato}
Consider $\A$, $\lambda$ and $L$ as in Lemma \ref{A-harm}. Then there exists $C_0\geq 1$ depending only on 
$n,N,\lambda$ and $L$ such that any $\A$-harmonic function $h$ on $B_{\rho/2}(x_0)$ satisfies
\begin{equation}
 \left(\frac{\rho}{2}\right)^2\sup_{{B_{\rho/4}}(x_0)}|Dh|^2
 +\left(\frac{\rho}{2}\right)^4\sup_{B_{\rho/4}(x_0)}|D^2h|^2
 \leq C_0\left(\frac{\rho}{2}\right)^2\dashint_{B_{\rho/2}(x_0)}|Dh|^2dx. \label{campanato}
\end{equation}
\end{thm}
We state the Poincar\'e inequality (Lemma \ref{Poincare}) in a convenient form. The proof can be founded in several literature, 
for example {\cite[Proposition 3.10]{GM}}.
\begin{lem}\label{Poincare}
There exists $C_P\geq 1$ depending only on $n$ such that every $u\in
W^{1,p}(B_\rho(x_0),\R^N)$ satisfies
\begin{equation}
 \int_{B_\rho(x_0)}|u-u_{x_0,\rho}|^pdx\leq
C_P\rho^p\int_{B_\rho(x_0)}|Du|^pdx.
 \label{poincare}
\end{equation}
\end{lem}
Given a function $u\in L^2(B_\rho(x_0),\R^N)$, where $x_0\in\R^n$ and $\rho>0$. 
We write $\ell_{x_0,\rho}$ for the minimizer of the functional
\begin{equation}
 \ell\mapsto \dashint_{B_\rho(x_0)}|u-\ell|^2dx \label{7}
\end{equation}
among all affine functions $\ell:\R^n\to\R^N$. 
Let write $\ell_{x_0,\rho}(x):=\ell_{x_0,\rho}(x_0)+D\ell_{x_0,\rho}(x-x_0)$.
It is easy to check that $\ell_{x_0,\rho}(x_0)=u_{x_0,\rho}$ and 
\begin{equation}
 D\ell_{x_0,\rho}=\frac{n+2}{\rho^2}\dashint_{B_\rho(x_0)}u\otimes (x-x_0)dx,
 \label{mini}
\end{equation}
where $\xi\otimes\zeta=\xi_i\zeta^\alpha$. Based on this formula, elementary calculations yield the following estimates.
\begin{lem}[{\cite[Lemma 2]{Kronz}}]
Assume $u\in L^2(B_\rho(x_0),\R^N)$, $x_0\in\R^n$, $\rho>0$ and $0<\theta\leq 1$. 
With $\ell_{x_0,\rho}$ and $\ell_{x_0,\theta\rho}$, we denote the affine functions from $\R^n$ to $\R^N$ defined as above for the 
radii $\rho$ and $\theta\rho$ respectively. Then we have
\begin{equation}
 |D\ell_{x_0,\rho}-D\ell_{x_0,\theta\rho}|^2
 \leq \frac{n(n+2)}{(\theta\rho)^2}\dashint_{B_{\theta\rho}(x_0)}
 |u-\ell_{x_0,\rho}|^2dx, \label{8}
\end{equation}
and more generally,  
\begin{equation}
 |D\ell_{x_0,\rho}-D\ell|^2
 \leq \frac{n(n+2)}{\rho^2}\dashint_{B_{\rho}(x_0)}|u-\ell|^2dx, \label{9}
\end{equation}
for all affine functions $\ell:\R^n\to\R^N$.
\end{lem}
The estimate \eqref{9} implies, in particular, that $\ell_{x_0,\rho}$ has the following quasi-minimizing property for 
the $L^p$-norm. The proof can be founded in \cite[Section 2]{BDHS}.
\begin{lem}\label{ell-minimize}
Consider the minimizer of \eqref{7}, that is, $\ell_{x_0,\rho}$. For any affine functions
$\ell:\R^n\to\R^N$ and $p\geq 2$ we have
\begin{equation}
 \dashint_{B_\rho(x_0)}|u-\ell_{x_0,\rho}|^pdx
 \leq c(n,p)\dashint_{B_\rho(x_0)}|u-\ell|^pdx .
\end{equation}
\end{lem}

Using Young's inequality, we obtain the following lemma. 

\begin{lem}\label{Young2}
Consider fixed $a,b\geq 0$, $p\geq 1$. Then for any $\varepsilon>0$, there exists
$K=K(p,\varepsilon)\geq 0$ satisfying
\begin{equation}
 (a+b)^p\leq (1+\varepsilon)a^p+Kb^p. \label{young2}
\end{equation}
\end{lem}
\begin{proof}
 We first consider the case $p=2k-1$ for $k\in \N$. By binomial theorem, we have
\begin{align*}
 (a+b)^{2k-1}
 &=\sum_{m=0}^{2k-1}\begin{pmatrix} 2k-1 \\ m \end{pmatrix}a^{2k-1-m}b^m\\
 &=a^{2k-1}+b^{2k-1}
 +\sum_{m=1}^{k-1}\begin{pmatrix} 2k-1 \\ m \end{pmatrix}
 (a^{2k-1-m}b^m+a^mb^{2k-1-m}).
\end{align*}
Using Young's inequality, we obtain
\[
 \sum_{m=1}^{k-1}\begin{pmatrix} 2k-1 \\ m \end{pmatrix}
 (a^{2k-1-m}b^m+a^mb^{2k-1-m})
 \leq\sum_{m=1}^{k-1}\begin{pmatrix} 2k-1 \\ m \end{pmatrix}
 (\varepsilon' a^{2k-1}+c(k,m,\varepsilon')b^{2k-1}),
\]
where $\varepsilon'>0$ will be fixed later. Thus, we get
\begin{align*}
(a+b)^{2k-1}
 &\leq a^{2k-1}+b^{2k-1}
 +\sum_{m=1}^{k-1}\begin{pmatrix} 2k-1 \\ m \end{pmatrix}
 (\varepsilon' a^{2k-1}+c(k,m,\varepsilon')b^{2k-1})\\
 &=\left\{ 1+\varepsilon'
 \sum_{m=1}^{k-1}\begin{pmatrix} 2k-1 \\ m \end{pmatrix}
 \right\} a^{2k-1}+\left\{ 1+
 \sum_{m=1}^{k-1}\begin{pmatrix} 2k-1 \\ m \end{pmatrix}
 c(k,m,\varepsilon')\right\}b^{2k-1}.
\end{align*}
For any $\varepsilon>0$ we conclude \eqref{Young2} by taking $\varepsilon'$ as $\varepsilon=\varepsilon'\displaystyle
 \sum_{m=1}^{k-1}\begin{pmatrix} 2k-1 \\ m \end{pmatrix}$. \\
\quad In case of $p=2k$, we may estimate similarly as above, hence we get
\begin{align*}
 (a+b)^{2k}
 &=\sum_{m=0}^{2k}\begin{pmatrix} 2k \\ m \end{pmatrix}a^{2k-m}b^m\\
 &=a^{2k}+b^{2k}
 +\sum_{m=1}^{k-1}\begin{pmatrix} 2k \\ m \end{pmatrix}
 (a^{2k-m}b^m+a^mb^{2k-m})
 +\begin{pmatrix} 2k \\ k \end{pmatrix}a^kb^k\\
 &\leq a^{2k}+b^{2k}
 +\sum_{m=1}^{k-1}\begin{pmatrix} 2k \\ m \end{pmatrix}
 (\varepsilon' a^{2k}+c(k,m,\varepsilon')b^{2k})
 +\begin{pmatrix} 2k \\ k \end{pmatrix}\left(\varepsilon'a^{2k}
 +\frac{1}{\varepsilon'}b^{2k}\right)\\
 &=\left\{ 1+\varepsilon'
 \sum_{m=1}^{k}\begin{pmatrix} 2k \\ m \end{pmatrix}
 \right\} a^{2k}+\left\{ 1+
 \sum_{m=1}^{k-1}\begin{pmatrix} 2k \\ m \end{pmatrix}
 c(k,m,\varepsilon')+\frac{1}{\varepsilon'}\right\}b^{2k}.
\end{align*}
This conclude that we have \eqref{Young2} for $p\in\N$. \\
\quad For general $p\geq 1$, let $[p]$ be the greatest integer not greater then $p$. We write
\[
 (a+b)^p
 =(a+b)^{[p]}(a+b)^{p-[p]}.
\]
By $0\leq p-[p]< 1$, we have
\[
 (a+b)^{p-[p]}\leq a^{p-[p]}+b^{p-[p]}.
\]
For $\varepsilon'>0$ to be fixed later, we get 
\[
 (a+b)^{[p]}\leq (1+\varepsilon')a^{[p]}+K(p,\varepsilon') b^{[p]},
\]
since $[p]\in\N$. Combining two estimates, we obtain
\begin{align*}
 (a+b)^p
 &\leq \left\{(1+\varepsilon')a^{[p]}+K(p,\varepsilon') b^{[p]}\right\}
 (a^{p-[p]}+b^{p-[p]})\\
 &=(1+\varepsilon')a^p+K(p,\varepsilon') b^p
 +(1+\varepsilon')a^{[p]}b^{p-[p]}+K(p,\varepsilon') a^{p-[p]}b^{[p]} \\
 &\leq (1+\varepsilon')a^p+K(p,\varepsilon') b^p
 +(1+\varepsilon'+K(p,\varepsilon'))
 (a^{[p]}b^{p-[p]}+a^{p-[p]}b^{[p]}).
\end{align*}
Again for $\varepsilon''>0$ to be fixed later, by using Young's inequality, we conclude
\[
 (a+b)^p\leq
 (1+\varepsilon')a^p+K(p,\varepsilon') b^p
 +(1+\varepsilon'+K(p,\varepsilon'))
 (\varepsilon''a^p+c(p,\varepsilon'')b^p).
\]
Take $\varepsilon'=\varepsilon/2$ and 
$\varepsilon''=\varepsilon'/(1+\varepsilon'+K(p,\varepsilon'))$, and this complete the proof.
\end{proof}
\begin{lem}[{\cite[Lemma 2.1]{GMo}}]\label{GM}
For $\delta \geq 0$, and for all $a,b\in\R^k$ we have 
\begin{equation}
 4^{-(1+2\delta)}\leq
 \frac{\displaystyle\int_0^1(1+|sa+(1-s)b|^2)^{\delta/2}ds}{(1+|a|^2+|b-a|^2)^{\delta/2}}
 \leq 4^\delta . \label{GM2}
\end{equation}
\end{lem}

\section{Proof of the main theorem}
To obtain the regularity result (Theorem \ref{pr}), we first prove Caccioppoli-type inequality. In the followings, we define $q>0$ as 
the dual exponent of $p\geq 2$, that is, $q=p/(p-1)$. Here we note that $q\leq 2$. 
\begin{lem}\label{Caccioppoli}
Let $u\in W^{1,p}(\Omega,\R^N)\cap L^{\infty}(\Omega,\R^N)$ be a bounded weak solution of the elliptic system 
\eqref{system} under the structure condition {\rm ({\bf H1}),({\bf H2}),({\bf H3}),({\bf H4})} and {\rm ({\bf H5})} 
with satisfying $\| u\|_{\infty}\leq M$ and $2^{(10-9p)/2}\lambda>a(M)M$. 
For any $ x_0\in\Omega$ and $\rho\leq 1$ with $B_{\rho}(x_0)\Subset\Omega$, 
and any affine functions $\ell:\R^n\to\R^N$ with $|\ell(x_0)|\leq M$, we have 
the estimate
\begin{align}
 \dashint_{B_{\frac{\rho}{2}}(x_0)}&\left\{\frac{|Du-D\ell|^2}{(1+|D\ell|)^2}
 +\frac{|Du-D\ell|^p}{(1+|D\ell|)^p}\right\} dx \notag\\
 \leq C_1&\Bigg[ \dashint_{B_\rho(x_0)}\left\{
 \frac{|u-\ell|^2}{\rho^2(1+|D\ell|)^2}+
 \frac{|u-\ell|^p}{\rho^p(1+|D\ell|)^p}\right\} dx \notag\\
 &+\omega \left(\dashint_{B_\rho(x_0)}|u-\ell(x_0)|^2dx \right)
 +V(\rho)+\left( a^q|D\ell|^q+b^q \right) \rho^q \Bigg], \label{caccioppoli}
\end{align}
with the constant $C_1=C_1(\lambda,p,L,a(M),M)\geq 1$. 
\end{lem}
\begin{proof}
Assume $x_0\in\Omega$ and $\rho\leq 1$ satisfy $B_{\rho}(x_0)\Subset\Omega$.
We take a standard cut-off function $\eta\in C^\infty_0(B_\rho(x_0))$
satisfying $0\leq\eta\leq 1$, $|D\eta|\leq 4/\rho$, $\eta\equiv 1$ on $B_{\rho/2}(x_0)$.
Then $\varphi :=\eta^p(u-\ell )$ is admissible as a test function in \eqref{ws}, and we obtain
\begin{align}
 \dashint_{B_\rho(x_0)}\eta^p\langle &A(x,u,Du),Du-D\ell \rangle dx \notag\\
 =-\,&\dashint_{B_\rho(x_0)}\langle A(x,u,Du),p\eta^{p-1}
 D\eta\otimes (u-\ell)\rangle dx
 +\dashint_{B_\rho(x_0)}\langle f,\varphi\rangle dx. \label{system2}
\end{align}
Furthermore, we have
\begin{align}
 -\,\dashint_{B_\rho(x_0)}\eta^p\langle &A(x,u,D\ell),Du-D\ell \rangle dx
\notag\\
 =\,&\dashint_{B_\rho(x_0)}\langle A(x,u,D\ell),p\eta^{p-1} D\eta\otimes
 (u-\ell)\rangle dx-\dashint_{B_\rho(x_0)}\langle A(x,u,D\ell),D\varphi
 \rangle dx, \label{13}
\end{align}
and 
\begin{equation}
 \dashint_{B_\rho(x_0)}\langle \left( A(\cdot ,\ell(x_0)
 ,D\ell)\right)_{x_0,\rho},D\varphi \rangle dx=0. \label{constcoeff}
\end{equation}
Adding \eqref{system2}, \eqref{13} and \eqref{constcoeff}, we obtain
\begin{align}
&\dashint_{B_\rho(x_0)}\eta^p \langle A(x,u,Du)-A(x,u,D\ell),
 Du-D\ell\rangle dx \notag \\
 =&-\dashint_{B_\rho(x_0)}\langle A(x,u,Du)-A(x,u,D\ell),
 p\eta^{p-1} D\eta\otimes (u-\ell)\rangle dx \notag \\
 &-\dashint_{B_\rho(x_0)}\langle
A(x,u,D\ell)-A(x,\ell(x_0),D\ell),D\varphi\rangle dx
 \notag \\
 &-\dashint_{B_\rho(x_0)}\langle A(x,\ell(x_0),D\ell)-\left(
 A(\cdot,\ell(x_0),D\ell)\right)_{x_0,\rho},D\varphi \rangle dx \notag \\
 &+\dashint_{B_\rho(x_0)}\langle f,\varphi\rangle dx \notag \\
 =:& \>\> \hbox{I}+\hbox{II}+\hbox{III}+\hbox{IV}. \label{caccio-divide}
\end{align}
The terms $\hbox{I},\hbox{II},\hbox{III},\hbox{IV}$ are defined above. 
Using the ellipticity condition ({\bf H2}) to the left-hand side of \eqref{caccio-divide}, we get
\begin{align}
 &\langle A(x,u,Du)-A(x,u,D\ell),Du-D\ell\rangle \notag\\
 =&\int_0^1\left\langle D_\xi A(x,u,sDu+(1-s)D\ell)(Du-D\ell),
 Du-D\ell\right\rangle ds \notag\\
 \geq& \lambda |Du-D\ell|^2\int_0^1(1+|sDu+(1-s)D\ell|)^{p-2}ds. \label{elliptic}
\end{align}
Then by using \eqref{GM2} in Lemma \ref{GM}, we obtain
\begin{align}
 &\langle A(x,u,Du)-A(x,u,D\ell),Du-D\ell\rangle \notag \\
 \geq& \lambda |Du-D\ell|^2\int_0^1
 (1+|sDu+(1-s)D\ell|^2)^{(p-2)/2}ds \notag\\
 \geq& 2^{(12-9p)/2}\lambda
 \left\{(1+|D\ell|)^{p-2}|Du-D\ell|^2+|Du-D\ell|^p\right\} . \label{elliptic2}
\end{align}
For $\varepsilon >0$ to be fixed later, using ({\bf H1}) and Young's inequality, we have
\begin{align}
 |\,\hbox{I}\,|
 \leq &\dashint_{B_\rho(x_0)}p\eta^{p-1}\left|\int_0^1D_\xi A(x,u,D\ell
 +s(Du-D\ell))(Du-D\ell)ds\right| |D\eta||u-\ell|dx \notag\\
 \leq &\dashint_{B_\rho(x_0)}
 c(p,L)\eta^{p-1}\left\{ (1+|D\ell|)^{p-2}+|Du-D\ell|^{p-2}\right\} |Du-D\ell||D\eta||u-\ell|dx \notag\\
 \leq &\varepsilon\dashint_{B_\rho(x_0)}\eta^p\left\{(1+|D\ell|)^{p-2}
 |Du-D\ell|^2+|Du-D\ell|^p\right\}dx \notag\\
 &+c(p,L,\varepsilon)\dashint_{B_\rho(x_0)}\left\{ (1+|D\ell|)^{p-2}\left|\frac{u-\ell}{\rho}\right|^2
 +\left|\frac{u-\ell}{\rho}\right|^p\right\}dx. \label{16}
\end{align}
In order to estimate $\hbox{II}$, we use ({\bf H3}), $D\varphi=\eta^p(Du-D\ell)+p\eta^{p-1}D\eta\otimes
(u-\ell)$, and again Young's inequality, we get 
\begin{align}
 |\,\hbox{II}\,|
 \leq &\varepsilon\dashint_{B_\rho(x_0)}\eta^p|Du-D\ell|^pdx
 +\varepsilon^{-q/p}\dashint_{B_\rho(x_0)}L^q\omega^q\left(|u-\ell(x_0)|^2\right)(1+|D\ell|)^p dx \notag\\
 &+\varepsilon\dashint_{B_\rho(x_0)}\left|\frac{u-\ell}{\rho}\right|^pdx
 +\varepsilon^{-q/p}\dashint_{B_\rho(x_0)}(4Lp)^q\omega^q\left(|u-\ell(x_0)|^2\right)(1+|D\ell|)^p dx \notag\\
 \leq &\varepsilon\dashint_{B_\rho(x_0)}\eta^p|Du-D\ell|^pdx
 +\varepsilon\dashint_{B_\rho(x_0)}\left|\frac{u-\ell}{\rho}\right|^pdx \notag\\
 &+c(p,L,\varepsilon)(1+|D\ell|)^p \omega\left(\dashint_{B_\rho(x_0)}|u-\ell(x_0)|^2dx\right), \label{17}
\end{align}
where we use Jensen's inequality in the last inequality. 
We next estimate $\hbox{III}$ by using the VMO-condition ({\bf H4}) and Young's inequality, we have
\begin{align*}
 |\,\hbox{III}\,|
 \leq & \frac{\varepsilon}{2^{p-1}}\dashint_{B_\rho(x_0)}
 \left\{\eta^p|Du-D\ell|+\frac{4p|u-\ell|}{\rho}\right\}^pdx
 +\left(\frac{2^{p-1}}{\varepsilon}\right)^{q/p}\dashint_{B_\rho(x_0)}
 {V_{x_0}}^q(x,\rho)(1+|D\ell|)^pdx.
\end{align*}
Then using the fact that ${V_{x_0}}^q={V_{x_0}}^{q-1}\cdot V_{x_0}\leq (2L)^{q-1}V_{x_0}\leq 2LV_{x_0}$, 
we infer
\begin{align}
 |\,\hbox{III}\,|\leq & \varepsilon\dashint_{B_\rho(x_0)}\eta^p|Du-D\ell|^pdx
 +c(p,\varepsilon)\dashint_{B_\rho(x_0)}\left|\frac{u-\ell}{\rho}\right|^pdx
 +c(p,L,\varepsilon)(1+|D\ell|)^pV(\rho). \label{18}
\end{align}
For $\varepsilon'>0$ to be fixed later, using ({\bf H5}), Lemma \ref{Young2} and Young's inequality, we have
\begin{align}
&|\,\hbox{IV}\,| \notag\\
 \leq & \dashint_{B_\rho(x_0)}a(|Du-D\ell|+|D\ell|)^p\eta^p|u-\ell|dx
 +\dashint_{B_\rho(x_0)}(b\eta\rho)\left|\frac{u-\ell}{\rho}\right|dx \notag\\
 \leq & \dashint_{B_\rho(x_0)}a\eta^p\left\{(1+\varepsilon')|Du-D\ell|^p
 +K(p,\varepsilon')|D\ell|^p\right\}|u-\ell|dx
 +\varepsilon b^q\rho^q +\varepsilon^{-p/q}\dashint_{B_\rho(x_0)}\left|\frac{u-\ell}{\rho}\right|^pdx \notag\\
 \leq & a(1+\varepsilon^\prime)(2M+|D\ell|\rho) \dashint_{B_\rho(x_0)}\eta^p|Du-D\ell|^pdx 
 +c(p,\varepsilon)\dashint_{B_\rho(x_0)} \left|\frac{u-\ell}{\rho}\right|^pdx \notag\\
 &+\varepsilon(1+|D\ell|)^p\rho^q\left\{ a^qK ^q |D\ell|^{q}+b^q\right\}. \label{19}
\end{align}
Combining \eqref{caccio-divide}, \eqref{elliptic2}, \eqref{17}, \eqref{18} and \eqref{19}, and set
$\lambda'=2^{(12-9p)/2}\lambda$C
$\Lambda :=\lambda'-3\varepsilon-a(1+\varepsilon')(2M+|D\ell|\rho)$, this gives
\begin{align}
 &\Lambda\dashint_{B_\rho(x_0)}\eta^p
 \left\{\frac{|Du-D\ell|^2}{(1+|D\ell|)^2}+\frac{|Du-D\ell|^p}{(1+|D\ell|)^p}\right\}dx \notag\\
 \leq & c(p,L,\varepsilon)\left[\dashint_{B_\rho(x_0)}
 \left\{\left|\frac{u-\ell}{\rho(1+|D\ell|)}\right|^2
 +\left|\frac{u-\ell}{\rho(1+|D\ell|)}\right|^p\right\}dx 
 +\omega\left(\dashint_{B_\rho(x_0)}|u-\ell(x_0)|^2dx\right) +V(\rho) \right] \notag\\
 &+\varepsilon\left\{ a^q(1+K(p,\varepsilon'))^q|D\ell|^q+b^q\right\}\rho^q. \label{roughcaccio}
\end{align}
Now choose $\varepsilon=\varepsilon(\lambda ,p,a(M),M)>0$ and $\varepsilon'=\varepsilon'(\lambda ,p,a(M),M)>0$ in a right way 
(for more precise way of choosing $\varepsilon$ and $\varepsilon'$, we refer to \cite[Lemma 4.1]{DG}), 
we obtain \eqref{caccioppoli}. 
\end{proof}

\begin{rem}\label{katei}
If we insert $p=2$ to the ``smallness condition'' $2^{(10-9p)/2}\lambda>a(M)M$, we obtain $\lambda/16>a(M)M$. 
On the other hand, we only need $\lambda/2>a(M)M$ to prove the Caccioppoli-type inequality {\rm (}Lemma \ref{Caccioppoli}{\rm )} 
since the term $(1+|sDu+(1-s)\nu|)^{p-2}$ in \eqref{elliptic} vanishes when $p=2$. This gap happens because 
the left-hand side inequality of \eqref{GM2} in Lemma \ref{GM}, which we used to estimate $(1+|sDu+(1-s)\nu|)^{p-2}$ 
from below, could not take equal when $\delta =p-2=0$. 
\end{rem}
To use the $\A$-harmonic approximation lemma, we need to estimate 
$\dashint_{B_\rho(x_0)}\A(D(u-\ell),D\varphi)dx$.
\begin{lem}\label{A-harm2}
Assume the same assumption in Lemma \ref{Caccioppoli}. Then for any $x_0\in\Omega$ 
and $\rho\leq \rho_0$ satisfy $B_{2\rho}(x_0)\Subset\Omega$, and any affine functions $\ell:\R^n\to\R^N$ 
with $|\ell(x_0)|\leq M$, the inequality
\begin{align}
 \dashint_{B_\rho(x_0)}\mathcal{A}(Dv,D\varphi )dx
 \leq C_2(1+|D\ell|)\biggl[ \mu^{1/2}&\left(\sqrt{\Psi_*(x_0,2\rho,\ell)}\right)
 \sqrt{\Psi_*(x_0,2\rho,\ell)} \notag\\
 &+\Psi_*(x_0,2\rho,\ell)+\rho(a|D\ell|^p+b)\biggr]
 \sup_{B_\rho(x_0)}|D\varphi| \label{Ah}
\end{align}
holds for all $\varphi\in C^\infty_0(B_\rho(x_0),\R^N)$ and a constant $C_2=C_2(n,\lambda,L,p,a(M))\geq 1$, where 
\begin{align*}
 \A(Dv,D\varphi):&=\frac{1}{(1+|D\ell|)^{p-1}}\left\langle
 \left( D_\xi A(\cdot,\ell(x_0),D\ell)\right)_{x_0,\rho}
 Dv,D\varphi\right\rangle , \\
 \Phi(x_0,\rho,\ell):&=\dashint_{B_\rho(x_0)}
 \left\{\frac{|Du-D\ell|^2}{(1+|D\ell|)^2}
 +\frac{|Du-D\ell|^p}{(1+|D\ell|)^p}\right\} dx, \\
 \Psi(x_0,\rho,\ell):&=\dashint_{B_\rho(x_0)}
 \left\{\frac{|u-\ell|^2}{\rho^2(1+|D\ell|)^2}
 +\frac{|u-\ell|^p}{\rho^p(1+|D\ell|)^p}\right\} dx, \\
 \Psi_*(x_0,\rho,\ell):&=\Psi(x_0,\rho,\ell)
 +\omega\left(\dashint_{B_\rho(x_0)}|u-\ell(x_0)|^2dx\right)+V(\rho)
 +\left( a^q|D\ell|^q+b^q\right)\rho^q, \\
 v:&=u-\ell=u-\ell(x_0)-D\ell(x-x_0).
\end{align*}
\end{lem}
\begin{proof}
Assume $x_0\in\Omega$ and $\rho\leq 1$ satisfy $B_{2\rho}(x_0)\Subset\Omega$.
Without loss of generality we may assume $\displaystyle\sup_{B_\rho(x_0)}|D\varphi|\leq 1$.
Note $\displaystyle\sup_{B_\rho(x_0)}|\varphi|\leq\rho\leq 1$. Using the fact that 
$\int_{B_\rho(x_0)}A(x_0,\xi,\nu)D\varphi dx=0$, we deduce
\begin{align}
 (1+|D\ell|)^{p-1}&\dashint_{B_\rho(x_0)}\mathcal{A}(Dv,D\varphi)dx \notag\\
 =&\dashint_{B_\rho(x_0)}\int_0^1
 \left\langle\left[\bigl( D_\xi A(\cdot,\ell(x_0),D\ell)\bigr)_{x_0,\rho}
 -\bigl( D_\xi A(\cdot,\ell(x_0),D\ell+sDv)\bigr)_{x_0,\rho}
 \right] Dv,D\varphi\right\rangle dsdx \notag\\
 &+\dashint_{B_\rho(x_0)}
 \left\langle\bigl(A(\cdot,\ell(x_0),Du)\bigr)_{x_0,\rho}-A(x,\ell(x_0),Du),
 D\varphi\right\rangle dx \notag\\
 &+\dashint_{B_\rho(x_0)}
 \langle A(x,\ell(x_0),Du)-A(x,u,Du),D\varphi\rangle dx \notag\\
 &+\dashint_{B_\rho(x_0)}\langle f,\varphi\rangle dx \notag\\
 =&:\hbox{I}+\hbox{II}+\hbox{III}+\hbox{IV} \label{21}
\end{align}
where terms $\hbox{I},\hbox{II},\hbox{III},\hbox{IV}$ are define above. 

Using the modulus of continuity $\mu$ from ({\bf H1}), Jensen's inequality and H\"{o}lder's inequality, we estimate 
\begin{align}
 |\,\hbox{I}\,|
 &\leq c(p,L)(1+|D\ell|)^{p-1}\dashint_{B_\rho(x_0)}\mu\left(\frac{|Du-D\ell|}{1+|D\ell|}\right)
 \left\{\frac{|Du-D\ell|}{1+|D\ell|}+\frac{|Du-D\ell|^{p-1}}{(1+|D\ell|)^{p-1}}\right\}dx \notag\\
 &\leq c\, (1+|D\ell|)^{p-1}\left[ \mu^{1/2}\left(\sqrt{\Phi(x_0,\rho,\ell)}\right)\sqrt{\Phi(x_0,\rho,\ell)}
 +\mu^{1/p}\left(\Phi^{1/2}(x_0,\rho,\ell)\right)\Phi^{1/q}(x_0,\rho,\ell)\right] \notag\\
 &\leq c\, (1+|D\ell|)^{p-1}\left[ \mu^{1/2}\left(\sqrt{\Phi(x_0,\rho,\ell)}\right) \sqrt{\Phi(x_0,\rho,\ell)}
 +\Phi(x_0,\rho,\ell)\right]. \label{22}
\end{align}
The last inequality follows from the fact that $a^{1/p}b^{1/q}=a^{1/p}b^{1/p}b^{(p-2)/p}\leq a^{1/2}b^{1/2}+b$ 
holds by Young's inequality. \\
\quad By using the VMO-condition, Young's inequality and the bound $V_{x_0}(x,\rho)\leq 2L$, 
the term $\hbox{II}$ can be estimated as 
\begin{align}
 |\,\hbox{II}\,|
 &\leq c(p)(1+|D\ell|)^{p-1}\dashint_{B_\rho(x_0)}
 \left\{V_{x_0}(x,\rho)+V_{x_0}(x,\rho)\frac{|Du-D\ell|^{p-1}}{(1+|D\ell|)^{p-1}}\right\}dx \notag\\
 &\leq c\, (1+|D\ell|)^{p-1} \left[\left( 1+(2L)^{p-1}\right) V(\rho)+\Phi(x_0,\rho,\ell)\right] . \label{23}
\end{align}
Similarly, we estimate the term $\hbox{III}$ by using the continuity condition ({\bf H3}), Young's inequality, the bound
$\omega\leq 1$ and Jensen's inequality. This leads us to 
\begin{align}
 |\,\hbox{III}\,|
 &\leq L\dashint_{B_\rho(x_0)}
 (1+|D\ell|+|Du-D\ell|)^{p-1}\omega\left(|u-\ell(x_0)|^2\right)dx \notag\\
 &\leq c(p,L)(1+|D\ell|)^{p-1} \left[\omega\left(\dashint_{B_\rho(x_0)}|u-\ell(x_0)|^2dx\right) +\Phi(x_0,\rho,\ell)\right] . 
 \label{24}
\end{align}
By using the growth condition ({\bf H5}) and $\displaystyle\sup_{B_\rho(x_0)}|\varphi|\leq\rho\leq 1$, we have
\begin{align}
 |\,\hbox{IV}\,|
 &\leq \dashint_{B_\rho(x_0)}\rho(a|Du|^p+b)dx \notag\\
 &\leq 2^{p-1}a(1+|D\ell|)^p\Phi(x_0,\rho,\ell)
 +2^{p-1}\rho(1+|D\ell|)^{p-1}(a|D\ell|^p+b). \label{25}
\end{align}
Combining \eqref{21} with the estimates \eqref{22}, \eqref{23}, \eqref{24} and \eqref{25},
we finally arrive at 
\begin{align*}
 \dashint_{B_\rho(x_0)}&\mathcal{A}(Dv,D\varphi)dx \\
 \leq & c(p,L,a(M))(1+|D\ell|) \notag\\
 &\times\left[ \mu^{1/2}\left(\sqrt{\Phi(x_0,\rho,\ell)}\right)
 \sqrt{\Phi(x_0,\rho,\ell)}+\Phi(x_0,\rho,\ell)
 +\Psi_*(x_0,\rho,\ell)+\rho(a|D\ell|^p+b)\right]\notag\\
 \leq& C_2(1+|D\ell|)
 \Biggl[ \mu^{1/2}\left(\sqrt{\Psi_*(x_0,2\rho,\ell)}\right)
 \sqrt{\Psi_*(x_0,2\rho,\ell)}
 +\Psi_*(x_0,2\rho,\ell)+\rho(a|D\ell|^p+b)\Biggr],
\end{align*}
where we use Caccioppoli-type inequality (Lemma \ref{Caccioppoli}), $\Phi(x_0,\rho,\ell)\leq
C_1\Psi_*(x_0,2\rho,\ell)$ and the concavity of $\mu$ to have $\mu(cs)\leq c\mu(s)$ for $c\geq 1$ at the last step.
\end{proof}
From now on, we write $\Phi(\rho)=\Phi(x_0,\rho,\ell_{x_0,\rho})$, 
$\Psi(\rho)=\Psi(x_0,\rho,\ell_{x_0,\rho})$, $\Psi_*(\rho)=\Psi_*(x_0,\rho,\ell_{x_0,\rho})$
for $x_0\in\Omega$ and $0<\rho\leq 1$. Here $\ell_{x_0,\rho}$ is a minimizer of (\ref{7}). 

Now we are in the position to establish the excess improvement. 
\begin{lem}\label{EI}
Assume the same assumptions with Lemma \ref{A-harm2}. Let $\theta\in (0,1/4]$ be arbitrary and impose the 
following smallness conditions on the excess:
\begin{enumerate}
 \item $\displaystyle{\mu^{1/2}\left( \sqrt{\Psi_*(\rho)}\right)
 +\sqrt{\Psi_*(\rho)}\leq\frac{\delta}{2}}$ with the constant 
 $\delta =\delta(n,N,p,\lambda,L,\theta^{n+p+2})$ from Lemma \ref{A-harm},
 \item $\displaystyle{\Psi(\rho)\leq\frac{\theta^{n+2}}{4n(n+2)}},$
 \item $\displaystyle{\gamma(\rho):=\left[{\Psi_*}^{q/2}(\rho)
 +\delta^{-q}\rho^q(a|D\ell_{x_0,\rho}|+b)^q\right]^{1/q}\leq 1}.$
\end{enumerate}
Then there holds the excess improvement estimate
\begin{equation}
 \Psi(\theta\rho)\leq C_3\theta^2\Psi_*(\rho) \label{EI-esti}
\end{equation}
with a constant $C_3\geq 1$ that depends only on $n$, $N$, $\lambda$, $L$, $p$, $a(M)$, $M$ and $\theta$. 
\end{lem}
\begin{proof}
We first rescale $u$ and set 
\[
 w:=\frac{u-\ell_{x_0,\rho}}{C_2(1+|D\ell_{x_0,\rho}|)\gamma}.
\]
We claim that $w$ satisfies the assumptions of Lemma \ref{A-harm}. By Lemma \ref{A-harm2}, with $\rho/2$
and $\ell_{x_0,\rho}$ instead of $\rho$ and $\ell$, and assumption (i), the map $w$ is approximately $\A$-harmonic
in the sense that 
\begin{align*}
 \dashint_{B_{\rho/2}(x_0)}\A(Dw,D\varphi)dx
 \leq& \left[\mu^{1/2}\left(\sqrt{\Psi_*(\rho)}\right)
 +\sqrt{\Psi_*(\rho)}+\frac{\delta}{2} \right] \sup_{B_{\rho/2}(x_0)}|D\varphi| \\
 \leq& \delta\sup_{B_{\rho/2}(x_0)}|D\varphi|,
\end{align*}
for all $\varphi\in C^\infty_0(B_{\rho/2}(x_0),\R^N)$, with the constant $\delta$ determined by Lemma \ref{A-harm} for the
choice $\varepsilon=\theta^{n+p+2}$.
Moreover, the choice of $C_2$, which implies $C_2\geq C_1$, and the Caccioppoli-type inequality (Lemma \ref{Caccioppoli}) infer
\[
 \dashint_{B_{\rho/2}(x_0)}\left\{|Dw|^2+\gamma^{p-2}|Dw|^p\right\}dx
 \leq \frac{C_1\Psi_*(\rho)}{{C_2}^2\gamma^2}
 \leq \frac{C_1}{{C_2}^2}
 \leq 1.
\]
Thus, Lemma \ref{A-harm} ensures the existence of an $\A$-harmonic map $h$ with the properties
\begin{align}
 \dashint_{B_{\rho/2}(x_0)}&\left\{\left|\frac{w-h}{\rho/2}\right|^2
 +\gamma^{p-2}\left|\frac{w-h}{\rho/2}\right|^p\right\} dx
 \leq\theta^{n+p+2} , \label{energybound}\\
 \dashint_{B_{\rho/2}(x_0)}&\left\{|Dh|^2
 +\gamma^{p-2}|Dh|^p\right\}dx\leq c(n,p).
\end{align}
Since $h$ is $\A$-harmonic, Theorem \ref{Campanato} yields the estimate for $s=2$ as well as for $s=p$
\[
 \sup_{B_{\rho/4}(x_0)}|D^2h|^s
 \leq c(s,n,N,p,\lambda ,L)\left(\frac{\rho}{2}\right)^{-s}.
\]
Therefore, using Taylor's theorem, we have the decay estimate, where $\theta\in(0,1/4]$ can be chosen arbitrarily: 
\begin{align*}
 &\gamma^{s-2}(\theta\rho)^{-s}\dashint_{B_{\theta\rho}(x_0)}
 |w-h(x_0)-Dh(x_0)(x-x_0)|^sdx \\
 \leq& 2^{s-1}\gamma^{s-2}(\theta\rho)^{-s}\left[\dashint_{B_{\theta\rho}(x_0)}|w-h|^sdx
 +\dashint_{B_{\theta\rho}(x_0)}|h(x)-h(x_0)-Dh(x_0)(x-x_0)|^sdx\right]\\
 \leq &c(s,n,N,p,\lambda ,L)\theta^2.
\end{align*}
Here we applied the energy bound \eqref{energybound} for the last estimate.
Scaling back to $u$ and using Lemma \ref{ell-minimize}, we conclude
\begin{align}
 &(\theta\rho)^{-s}\dashint_{B_{\theta\rho}(x_0)}|u-\ell_{x_0,\theta\rho}|^sdx \notag\\
 \leq &c(n,s)(\theta\rho)^{-s}\dashint_{B_{\theta\rho}(x_0)}
 |u-\ell_{x_0,\rho}-C_2\gamma(1+|D\ell_{x_0,\rho}|)
 \left( h(x_0)+Dh(x_0)(x-x_0)\right)|^sdx \notag\\
 =&c(s,n,N,p,\lambda ,L,a(M))(\theta\rho)^{-s}\gamma^s(1+|D\ell_{x_0,\rho}|)^s
 \dashint_{B_{\theta\rho}(x_0)}|w-h(x_0)-Dh(x_0)(x-x_0)|^sdx \notag\\
 \leq &c\, \gamma^2 (1+|D\ell_{x_0,\rho}|)^s\theta^2 \notag\\
 \leq &c\, (1+|D\ell_{x_0,\rho}|)^s\theta^2\left[{\Psi_*}^{q/2}(\rho)
 +2^{q/p}\delta^{-q}\Psi_*(\rho)\right]^{2/q} \notag\\
 \leq &c\, (1+|D\ell_{x_0,\rho}|)^s\theta^2\Psi_*(\rho). \label{ei-rough}
\end{align}
Here we would like to replace the term $|D\ell_{x_0,\rho}|$ on the right-hand side by $|D\ell_{x_0,\theta\rho}|$. 
For this, we use (\ref{8}) and the assumption (ii) in order to estimate
\begin{align*}
 |D\ell_{x_0,\rho}-D\ell_{x_0,\theta\rho}|^2
 \leq&\frac{n(n+2)}{(\theta\rho)^2}\dashint_{B_{\theta\rho}(x_0)}
 |u-\ell_{x_0,\rho}|^2dx \\
 \leq& \frac{n(n+2)}{\theta^{n+2}\rho^2}\dashint_{B_\rho(x_0)}
 |u-\ell_{x_0,\rho}|^2dx \\
 \leq& \frac{n(n+2)}{\theta^{n+2}}(1+|D\ell_{x_0,\rho}|)^2\Psi(\rho)
 \leq \frac{1}{4}(1+|D\ell_{x_0,\rho}|)^2.
\end{align*}
This yields 
\[
1+|D\ell_{x_0,\rho}|\leq1+|D\ell_{x_0,\theta\rho}|+|D\ell_{x_0,\rho}-D\ell_{x_0,\theta\rho}|
\leq1+|D\ell_{x_0,\theta\rho}|+\frac{1}{2}(1+|D\ell_{x_0,\rho}|),
\] 
and after reabsorbing the last term from the right-hand side on the left, we also obtain
\[
 1+|D\ell_{x_0,\rho}|\leq 2(1+|D\ell_{x_0,\theta\rho}|).
\]
Plugging this into \eqref{ei-rough}, we deduce
\[
 (\theta\rho)^{-s}\dashint_{B_{\theta\rho}(x_0)}|u-\ell_{x_0,\theta\rho}|^sdx
 \leq c(s,n,N,p,\lambda ,L,a(M))(1+|D\ell_{x_0,\theta\rho}|)^s\theta^2\Psi_*(\rho)
\]
for $s=2$ and $s=p$. Dividing by $(1+|D\ell_{x_0,\theta\rho}|)^s$, then adding the corresponding
terms for $s=2$ and $s=p$, we deduce the claim. 
\end{proof}

We fix an arbitrarily H\"{o}lder exponent $\alpha\in(0,1)$ and define the Campanato-type excess
\[
 C_\alpha(x_0,\rho):=\rho^{-2\alpha}
 \dashint_{B_\rho(x_0)}|u-u_{x_0,\rho}|^2dx.
\]
In the following lemma, we iterate the excess improvement estimate \eqref{EI-esti} from Lemma \ref{EI} and 
obtain the bounededness of the two excess functionals, $C_\alpha$ and $\Psi$. 
\begin{lem}\label{iteration}
Under the same assumption with Lemma \ref{EI}, for every $\alpha\in(0,1)$, there exists constants 
$\varepsilon_*,\kappa_*,\rho_*>0$ and $\theta\in(0,1/8]$, all depending at most on 
$n$, $N$, $\lambda$, $p$, $L$, $\alpha$, $\rho_0$, $\mu(\cdot)$, $\omega(\cdot)$, $V(\cdot)$, $a(M)$, $b$ and $M$, 
such that the conditions 
\begin{equation}
 \Psi(\rho)<\varepsilon_*,\quad\text{and}\quad
 C_\alpha(x_0,\rho)<\kappa_* \tag{$A_0$}
\end{equation}
for all $\rho\in(0,\rho_*)$ with $B_\rho(x_0)\Subset\Omega$, imply
\begin{equation}
 \Psi(\theta^k\rho)<\varepsilon_*,\quad\text{and}\quad
 C_\alpha(x_0,\theta^k\rho)<\kappa_* \tag{$A_k$}
\end{equation}
respectively, for every $k\in\N$. 
\end{lem}
\begin{proof}
We begin by choosing the constants. First, let 
\[
 \theta:=\min\left\{\left(\frac{1}{16n(n+2)}\right)^{1/(2-2\alpha)},
 \frac{1}{\sqrt{4C_3}} \right\}\leq\frac{1}{8},
\]
with the constant $C_3$ determined in Lemma \ref{A-harm2}. In particular, the choice of $\theta=\theta(n,N,\lambda,L,a,M,\alpha)>0$
fixes the constant $\delta=\delta(n,N,\lambda,L,a,M,\alpha)>0$ from Lemma \ref{A-harm}. 
Next, we fix an $\varepsilon_*=\varepsilon_*(n,N,\lambda,L,a,M,\alpha,\mu(\cdot))>0$ sufficiently small to ensure
\[
 \varepsilon_*\leq\frac{\theta^{n+2}}{16n(n+2)}\quad
 \text{and}\quad \mu^{1/2}\left(\sqrt{4\varepsilon_*}\right)
 +\sqrt{4\varepsilon_*}\leq\frac{\delta}{2}.
\]
Then, we choose $\kappa_*=\kappa_*(n,N,\lambda,L,a,M,\alpha,\mu(\cdot),\omega(\cdot))>0$ so small that
\[
 \omega(\kappa_*)<\varepsilon_*.
\]
Finally, we fix $\rho_*=\rho_*(n,N,\lambda,p,L,\alpha,\rho_0,\mu(\cdot),\omega(\cdot),V(\cdot),a,b,M)>0$ small
enough to guarantee
\[
 \rho_*\leq\min\{\rho_0,{\kappa_*}^{1/(2-2\alpha)},1\}, 
 \quad V(\rho_*)<\varepsilon_*\quad\text{and}\quad
 \left\{\left( a\sqrt{n(n+2)\kappa_*}\right)^q+b^q\right\}{\rho_*}^{q\alpha}<\varepsilon_*.
\]
Now we prove the assertion ($A_k$) by induction. We assume that we have already established ($A_k$) up to some 
$k\in\N\cup\{0\}$. We begin with proving the first part of the assertion ($A_{k+1}$), that is, the one concerning 
$\Psi(\theta^{k+1}\rho)$. First, using \eqref{9} with $\ell\equiv u_{x_0,\theta^k\rho}$, we obtain
\begin{align}
 |D\ell_{x_0,\theta^k\rho}|^2&\leq\frac{n(n+2)}{(\theta^k\rho)^2}
 \dashint_{B_{\theta^k\rho}(x_0)}|u-u_{x_0,\theta^k\rho}|^2dx \notag\\
 &=n(n+2)(\theta^k\rho)^{2\alpha-2}C_\alpha(x_0,\theta^k\rho) \notag\\
 &\leq n(n+2){\rho_*}^{2\alpha-2}\kappa_*. \label{dell}
\end{align}
Thus, the assumption ($A_k$), the choice of $\kappa_*$ and $\rho_*$, and the above estimate infer
\begin{align}
 \Psi_*(\theta^k\rho)&\leq\Psi(\theta^k\rho)
 +\omega(C_\alpha(x_0,\theta^k\rho))+V(\theta^k\rho)
 +(a^q|D\ell_{x_0,\theta^k\rho}|^q+b^q)(\theta^k\rho)^q \notag\\
 &\leq\varepsilon_*+\omega(\kappa_*)+V(\rho_*)
 +\left(\left(a\sqrt{n(n+2)\kappa_*}\right)^q
 +b^q\right){\rho_*}^{q\alpha}<4\varepsilon_*. \label{psiepsi}
\end{align}
Now it is easy to check that our choice of $\varepsilon_*$ implies that the smallness condition assumptions (i) and (ii)
in Lemma \ref{EI} are satisfied on the level $\theta^k\rho$, that is, we have
\begin{equation}
 \mu^{1/2}\left(\sqrt{\Psi_*(\theta^k\rho)}\right)
 +\sqrt{\Psi_*(\theta^k\rho)}
 <\mu^{1/2}\left(\sqrt{4\varepsilon_*}\right)+\sqrt{4\varepsilon_*}
 \leq\frac{\delta}{2}, \label{sm1}
\end{equation}
and 
\begin{equation}
 \Psi(\theta^k\rho)<\varepsilon_*
 <\frac{\theta^{n+2}}{4n(n+2)}. \label{sm2}
\end{equation}
Furthermore, we have the smallness condition assumption (iii), that is,  
\begin{equation}
 \gamma(\theta^k\rho)=\left[{\Psi_*}^{q/2}(\theta^k\rho)+\delta^{-q}
 (\theta^k\rho)^q(a|D\ell_{x_0,\theta^k\rho}|+b)^q\right]^{1/q}\leq 1. \label{sm3}
\end{equation}
To check \eqref{sm3}, first, note that $\Psi_*(\theta^k\rho)<1$ holds by the estimate \eqref{psiepsi} and 
the choice of $\varepsilon_*$. This implies 
\begin{equation}
 {\Psi_*}^{q/2}(\theta^k\rho)\leq {\Psi_*}^{1/2}(\theta^k\rho)
 <\sqrt{4\varepsilon_*}\leq \frac{\delta}{4}. \label{gamma1st}
\end{equation}
Next, using \eqref{dell} and ${\rho_*}^{\alpha-1}\geq 1$, we obtain 
\begin{align*}
 \delta^{-q}(\theta^k\rho)^q(a|D\ell|_{x_0,\theta^k\rho}+b)^q
 \leq&\delta^{-q}{\rho_*}^q(a\sqrt{n(n+2)\kappa_*}{\rho_*}^{\alpha-1}+b)^q \\
 \leq&\delta^{-q}{\rho_*}^{q\alpha}(a\sqrt{n(n+2)\kappa_*}+b)^q\\
 \leq&\delta^{-q}{\rho_*}^{q\alpha}2^{q/p}
 \left\{\left(a\sqrt{n(n+2)\kappa_*}\right)^q+b^q\right\}. 
\end{align*}
Then the choice of $\rho_*$ and $\varepsilon_*$ imply
\begin{equation}
 \delta^{-q}(\theta^k\rho)^q(a|D\ell|_{x_0,\theta^k\rho}+b)^q
 \leq\delta^{-q}2^{q/p}\varepsilon_*
 \leq 2^{-4+q/p}\delta^{2-q}
 \leq\frac{\delta}{8}. \label{gamma2nd}
\end{equation}
Therefore combining \eqref{gamma1st} and \eqref{gamma2nd}, we have \eqref{sm3}. 
We may thus apply Lemma \ref{EI} with the radius $\theta^k\rho$ instead of $\rho$, which yields 
\[
 \Psi(\theta^{k+1}\rho)\leq C_3\theta^2\Psi_*(\theta^k\rho)
 <4C_3\theta^2\varepsilon_*\leq\varepsilon_*,
\]
by the choice of $\theta$. We have thus established the first part of the assertion ($A_{k+1}$) and it remains to 
prove the second one, that is, the one concerning $C_\alpha(x_0,\theta^{k+1}\rho)$. For this aim, we first compute
\[
 \frac{1}{(\theta^k\rho)^2}\dashint_{B_{\theta^k\rho}(x_0)}
 |u-\ell_{x_0,\theta^k\rho}|^2dx
 \leq (1+|D\ell_{x_0,\theta^k\rho}|)^2\Psi(\theta^k\rho)
 \leq 2\varepsilon_*+2\varepsilon_*|D\ell_{x_0,\theta^k\rho}|^2
\]
where we used the assumption ($A_k$) in the last step. Since 
$\ell_{x_0,\theta^k\rho}(x)=u_{x_0,\theta^k\rho}+D\ell_{x_0,\theta^k\rho} (x-x_0)$, we can estimate
\begin{align*}
 C_\alpha(x_0,\theta^{k+1}\rho)&\leq
 (\theta^{k+1}\rho)^{-2\alpha}\dashint_{B_{\theta^{k+1}\rho}(x_0)}
 |u-u_{x_0,\theta^k\rho}|^2dx \\
 &\leq 2(\theta^{k+1}\rho)^{-2\alpha}\left[\dashint_{B_{\theta^{k+1}\rho}(x_0)}
 |u-\ell_{x_0,\theta^k\rho}|^2dx+|D\ell_{x_0,\theta^k\rho}|^2(\theta^{k+1}\rho)^2\right]\\
 &\leq 2(\theta^{k+1}\rho)^{-2\alpha}\left[\theta^{-n}\dashint_{B_{\theta^k\rho}(x_0)}
 |u-\ell_{x_0,\theta^k\rho}|^2dx+|D\ell_{x_0,\theta^k\rho}|^2(\theta^{k+1}\rho)^2\right]\\
 &\leq 4(\theta^k\rho)^{2-2\alpha}\left[\varepsilon_*\theta^{-n-2\alpha}
 +|D\ell_{x_0,\theta^k\rho}|^2(\varepsilon_*\theta^{-n-2\alpha}+\theta^{2-2\alpha})\right] .
\end{align*}
Using \eqref{dell} and recalling the choice of $\rho_*$, $\varepsilon_*$ and $\theta$, we deduce
\begin{align*}
 C_\alpha(x_0,\theta^{k+1}\rho)
 &\leq 4{\rho_*}^{2-2\alpha}\left[\varepsilon_*\theta^{-n-2\alpha}+n(n+2)
 \kappa_*{\rho_*}^{2-2\alpha}(\varepsilon_*\theta^{-n-2\alpha}+\theta^{2-2\alpha})\right]\\
 &\leq\frac{1}{4}{\rho_*}^{2-2\alpha}\theta^{2-2\alpha}
 +8n(n+2)\kappa_*\theta^{2-2\alpha} \\
 &\leq\frac{1}{4}\kappa_*+\frac{1}{2}\kappa_*<\kappa_*.
\end{align*}
This proves the second part of the assertion ($A_{k+1}$) and finally we conclude the proof of the lemma. 
\end{proof}

Now, to obtain the regularity result (Theorem \ref{pr}), it is similar arguments as in \cite[Section 3.5]{BDHS} by using 
Lemma \ref{iteration}. \\
\mbox{}\\
{\bf Acknowledgments}\\
This results are obtained as the Master thesis in Tokyo University of Science under advising by Professor 
Atsushi Tachikawa. The author thanks Professor Atsushi Tachikawa for helpful discussions.

\def\cprime{$'$}

\providecommand{\bysame}{\leavevmode\hbox to3em{\hrulefill}\thinspace}

\providecommand{\MR}{\relax\ifhmode\unskip\space\fi MR }

\providecommand{\MRhref}[2]{%
  
\href{http://www.ams.org/mathscinet-getitem?mr=#1}{#2}
}

\providecommand{\href}[2]{#2}

\mbox{}\\
Taku Kanazawa \\
Graduate School of Mathematics \\
Nagoya University \\
Chikusa-ku, Nagoya, 464-8602, JAPAN \\
taku.kanazawa@math.nagoya-u.ac.jp
\end{document}